\documentclass[a4paper,11pt]{article}
\usepackage[T1]{fontenc}
\usepackage{amsfonts}
\usepackage{url, textcomp}
\usepackage{amsmath, amsthm, amssymb, mathrsfs}
\usepackage{graphicx}
\usepackage{array} 
\usepackage{bm}

\usepackage{tabularx}
\usepackage{multirow}
\usepackage{dcolumn}

\usepackage[round]{natbib}
\title{Rejection Sampling for Tempered L\'evy Processes}
\author{Michael Grabchak\footnote{Email address: mgrabcha@uncc.edu}\ \ 
{\it University of North Carolina Charlotte}}

\begin{document}
\newtheorem{prop}{Proposition}
\newtheorem{thrm}{Theorem}
\newtheorem{defn}{Definition}
\newtheorem{cor}{Corollary}
\newtheorem{lemma}{Lemma}
\newtheorem{remark}{Remark}
\newtheorem{exam}{Example}

\newcommand{\rd}{\mathrm d}
\newcommand{\rE}{\mathrm E}
\newcommand{\ts}{TS^p_\alpha}
\newcommand{\tr}{\mathrm{tr}}
\newcommand{\iid}{\stackrel{\mathrm{iid}}{\sim}}
\newcommand{\eqd}{\stackrel{d}{=}}
\newcommand{\cond}{\stackrel{d}{\rightarrow}}
\newcommand{\conv}{\stackrel{v}{\rightarrow}}
\newcommand{\conw}{\stackrel{w}{\rightarrow}}
\newcommand{\conp}{\stackrel{p}{\rightarrow}}
\newcommand{\simp}{\stackrel{p}{\sim}}

\maketitle

\begin{abstract}
We extend the idea of tempering stable L\'evy processes to tempering more general classes of L\'evy processes. We show that the original process can be decomposed into the sum of the tempered process and an independent point process of large jumps. We then use this to set up a rejection sampling algorithm for sampling from the tempered process. A small scale simulation study is given to help understand the performance of this algorithm.\\

\end{abstract}

\section{Introduction}

Tempered stable distributions are a class of models obtained by modifying the tails of stable distributions to make them lighter. This leads to models that are more realistic for a variety of applications, where real-world frictions prevent extremely heavy tails from occurring. Perhaps the earliest models of this type are Tweedie distributions, which were introduced in \cite{Tweedie:1984}, see also \cite{Kuchler:Tappe:2013} for a recent review. A more general approach is given in \cite{Rosinski:2007}. This was further generalized in several directions in \cite{Rosinski:Sinclair:2010}, \cite{Bianchi:Rachev:Kim:Fabozzi:2011}, and \cite{Grabchak:2012}. A survey with many references, which discuss a variety of applications, including those to actuarial science, biostatistics, computer science, mathematical finance, and physics can be found in \cite{Grabchak:2016book}. Associated with every tempered stable distribution, is a tempered stable L\'evy process, which behaves like a stable L\'evy process in a small time frame, but it has fewer large jumps.

The purpose of this paper is two-fold. First, we extend the idea of tempering a stable L\'evy process to tempering any L\'evy process, and give results about the relationship between the original process and the tempered one. In particular, we show that the original process can be decomposed into the sum of the tempered process and an independent point process of large jumps.  Our second purpose is to use this decomposition to set up a rejection sampling algorithm for simulating from the tempered process.

The problem of simulation has not been resolved even for tempered stable distributions. For these, rejection sampling techniques are currently known only for Tweedie distributions. For other tempered stable distributions, the only known exact simulation technique is the inversion method, which is computationally inefficient because it requires numerically calculating the quantile function, which can only be done by numerically inverting the cumulative distribution function (cdf). However, calculating the cdf is, itself, expensive, since, for most tempered stable distributions, it can only be evaluated by numerically applying an inverse Fourier transform to the characteristic function. Other simulation techniques are only approximate, and are based either on truncating a shot-noise representation  \citep{Imai:Kawai:2011}, or on approximations by a compound Poisson Process \citep{Baeumer:Kovacs:2012} or a Brownian motion \citep{Cohen:Rosinski:2007}.

The basic idea of our rejection sampling approach is to start by sampling an increment of the original process. This increment is then rejected if it is too large, otherwise it is accepted. The procedure is in keeping with the motivation for defining tempered processes as having marginal distributions that are similar to those of the original process, but with lighter tails. In deciding if the observation is too large, we require the ability to evaluate the probability density functions of the marginal distributions of both the original and the tempered process. While this may be computationally challenging, in many situations it is more efficient than implementing the inversion method, see the simulation results in Section \ref{sec: sims} below.

The rest of the paper is organized as follows. In Section \ref{sec: inf div}, we recall basic facts about infinitely divisible distributions and their associated L\'evy processes. In Section \ref{sec: main}, we formally introduce tempered L\'evy processes and give our main theoretical results. Then, in Section \ref{sec: sampling}, we show how to use these results to set up a rejection sampling algorithm. In Section \ref{sec: ts}, we recall some basic facts about tempered stable distributions and give conditions under which our results hold. Then, in Section \ref{sec: pts}, we give detailed conditions for an important subclass of tempered stable distributions, which has additional structure and is commonly used. Finally, in Section \ref{sec: sims}, we give a small scale simulation study to illustrate how our method works in practice.

Before proceeding, we introduce some notation. Let $\mathbb N=\{1,2,\dots\}$ be the set of natural numbers. Let $\mathbb R^d$ be the space of $d$-dimensional column vectors of real numbers equipped with the usual inner product $\langle\cdot,\cdot\rangle$ and the usual norm $|\cdot|$. 
Let $\mathbb S^{d-1}=\{x\in\mathbb R^d: |x|=1\}$ denote the unit sphere in $\mathbb R^d$. Let $\mathfrak B(\mathbb R^d)$ and $\mathfrak B(\mathbb S^{d-1})$ denote the Borel sets in $\mathbb R^d$ and $\mathbb S^{d-1}$, respectively. For a Borel measure $M$ on $\mathbb R^d$ and $t\ge0$, we write $tM$ to denote the Borel measure on $\mathbb R^d$ given by $(tM)(B)=tM(B)$ for every $B\in\mathfrak B(\mathbb R^d)$. 
If $a,b\in\mathbb R$, we write $a\vee b$ and $a\wedge b$ to denote, respectively, the maximum and the minimum of $a$ and $b$. If $\mu$ is a probability measure on $\mathbb R^d$, we write $X\sim\mu$ to denote that $X$ is an $\mathbb R^d$-valued random variable with distribution $\mu$. Finally, we write $U(0,1)$ to denote the uniform distribution on $(0,1)$ and $N(b,A)$ to denote the multivariate normal distribution with mean vector $b$ and covariance matrix $A$. 

\section{Infinitely Divisible Distributions and L\'evy Processes}\label{sec: inf div}

Recall that an infinitely divisible distribution $\mu$ on $\mathbb R^d$ is a probability measure with a characteristic function of the form $\phi_\mu(z) = \exp\{C_\mu(z)\}$, where, for $z\in\mathbb R^d$,
\begin{eqnarray*}
C_\mu(z) = -\frac{1}{2}\langle z,Az\rangle + i\langle b,z\rangle + \int_{\mathbb R^d}\left(e^{i\langle z,x\rangle}-1-i\langle z,x\rangle h(x)\right)L(\rd x).
\end{eqnarray*}
Here, $A$ is a symmetric nonnegative-definite $d\times d$-dimensional matrix called the Gaussian part, $b\in\mathbb R^d$ is called the shift, and $L$ is a Borel measure, called the L\'evy measure, which satisfies
\begin{eqnarray}\label{eq: cond for levy measure}
L(\{0\})=0 \mbox{\ and\ } \int_{\mathbb R^d}(|x|^2\wedge1)L(\rd x)<\infty.
\end{eqnarray}
The function $h:\mathbb R^d\mapsto\mathbb R$, which we call the $h$-function, can be any Borel function satisfying
\begin{eqnarray}\label{eq: when use h}
\int_{\mathbb R^d}\left|e^{i\langle z,x\rangle}-1-i\langle z,x\rangle h(x)\right|L(\rd x)<\infty
\end{eqnarray}
for all $z\in\mathbb R^d$. For a fixed $h$-function, the parameters $A$, $L$, and $b$ uniquely determine the distribution $\mu$, and we write
$$
\mu=ID(A,L,b)_h.
$$
The choice of $h$ does not affect parameters $A$ and $L$, but different choices of $h$ result in different values for $b$, see Section 8 in \cite{Sato:1999}. 

Associated with every infinitely divisible distribution $\mu=ID(A,L,b)_h$ is a L\'evy process, $\{X_t:t\ge0\}$, which is stochastically continuous with independent and stationary increments, and the characteristic function of $X_t$ is $\left(\phi_\mu(z)\right)^t$. We denote the distribution of $X_t$ by $\mu^t$. It follows that, for each $t\ge0$, $X_t\sim\mu^t= ID(tA,tL,tb)_h$. For more on infinitely divisible distributions and their associated L\'evy processes see \cite{Sato:1999}.

\section{Main Results}\label{sec: main}

Let $\mu=ID(A,L,b)_h$ be any infinitely divisible distribution,
and define the Borel measure $\tilde L$ by
$$
\tilde L(\rd x) = g(x) L(\rd x),
$$
where $g:\mathbb R^{d}\mapsto[0,\infty)$ is a Borel function. Throughout, we make the following assumptions: 
\begin{itemize}
\item[\textbf{A1.}] $0\le g(x) \le 1$ for all $x\in\mathbb R^{d}$, and
\item[\textbf{A2.}]
\begin{eqnarray*}
\int_{\mathbb R^d}\left(\left|xh(x)\right|\vee1\right) \left(1-g(x)\right) L(\rd x)<\infty.
\end{eqnarray*}
\end{itemize}

Assumption A1 guarantees that $\tilde L$ satisfies \eqref{eq: cond for levy measure} and \eqref{eq: when use h}. Thus $\tilde L$ is a valid L\'evy measure, and we can use the same $h$-function with $\tilde L$ as with $L$. Let $\tilde\mu=ID(A,\tilde L,\tilde b)_h$, where
$$
\tilde b = b - \int_{\mathbb R^{d}} x h(x)\left(1-g(x)\right)  L(\rd x).
$$
We call $\tilde \mu$ the tempering of $\mu$ and we call $g$ the tempering function.

\begin{remark}
The name ``tempering function'' comes from the fact that, when the additional assumption $\lim_{|x|\to\infty}g(x)=0$ holds, the tails of the distribution $\tilde\mu$ are lighter than those of $\mu$. In this sense, the distribution $\tilde\mu$ ``tempers'' the tails of $\mu$. While this assumption is part of the motivation for defining such distributions, we do not require it in this paper.
\end{remark}

\begin{remark}
We can always take $h(x) = 1_{[|x|\le1]}$. In this case $\left(\left|xh(x)\right|\vee1\right)=1$ and Assumption A2 becomes $\int_{\mathbb R^d}\left(1-g(x)\right) L(\rd x)<\infty$. However, when we work with other $h$-functions, we need Assumption A2 to be as given.
\end{remark}

In light of Assumptions A1 and A2, we can define the finite Borel measure
$$
\rho(\rd x) = \left(1-g(x)\right) L(\rd x).
$$
Now, set
$$
\eta := \rho(\mathbb R^d)<\infty,
$$
and define the probability measure
$$
\rho_1(B) = \frac{\rho(B)}{\eta},  \ \ B\in\mathfrak B(\mathbb R^d).
$$

Let $Z_1,Z_2,\dots$ be independent and identically distributed (iid) random variables with distribution $\rho_1$. Independent of these, let $\{N_t:t\ge0\}$ be a Poisson process with intensity $\eta$, and set
\begin{eqnarray}\label{eq: Vt}
V_t = \sum_{i=1}^{N_t} Z_i, \ \ t\ge0.
\end{eqnarray}
This is a compound Poisson process and, by Proposition 3.4 in \cite{Cont:Tankov:2004}, the characteristic function of $V_t$ is given by
$$
\rE\left[e^{i\langle V_t,z\rangle}\right] = \exp\left\{ t\int_{\mathbb R^d} \left(e^{i\langle x,z\rangle}-1\right) \rho(\rd x)\right\}, \  \ z\in\mathbb R^d.
$$
Let $T = \inf\{t: N_t>0\}$ and note that, by properties of Poisson processes, $T$ has an exponential distribution with rate $\eta$, i.e.\
\begin{eqnarray}\label{eq: distrib of T}
P(T>t) = e^{-t\eta}, \ \ t>0.
\end{eqnarray}
We now give our main result, which generalizes a result about relativistic stable distributions\footnote{Relativistic stable distributions are the distributions of $\sqrt X Z$, where $X$ and $Z$ are independent, $X$ has a Tweedie distribution, and $Z\sim N(0,I)$, where $I$ is the identity matrix.} given in \cite{Ryznar:2002}.

\begin{thrm}\label{thrm: repres}
Let $\mu=ID(A,L,b)_h$ and let $g:\mathbb R^{d}\mapsto[0,\infty)$ be a Borel function satisfying Assumptions A1 and A2. Let $\tilde\mu=ID(A,\tilde L,\tilde b)_h$, $V=\{V_t:t\ge0\}$, and $T$ be as described above. Let $\tilde X=\{\tilde X_t:t\ge0\}$ be a L\'evy process, independent of $V$, with $\tilde X_1\sim\tilde\mu$ and set
$$
X_t = \tilde X_t + V_t, \ \ \ t\ge0.
$$ 
1. The process $X=\{X_t:t\ge0\}$ is a L\'evy process with $X_1\sim \mu$.\\
2. If $0\le t<T$, then $X_t=\tilde X_t$.\\
3. For any $B\in\mathfrak B(\mathbb R^d)$ and $t\ge0$ we have
$$
P(\tilde X_t\in B) \le e^{t\eta} P(X_t\in B).
$$
\end{thrm}

This theorem implies that the process $\tilde X$ is obtained from $X$ by throwing out the jumps that are governed by $V_t$. We call $\tilde X$ the tempered L\'evy process, and, in this context, we refer to $X$ as the original process.

\begin{proof}
We begin with the first part. Since $\tilde X$ and $V$ are independent L\'evy processes and the sum of independent L\'evy processes is still a L\'evy process, it suffices to check that the characteristic function of $X_t$ is the same as the characteristic function of $\tilde X_t + V_t$. The characteristic function of $\tilde X_t+V_t$ can be written as $e^{t C(z)}$, where for $z\in\mathbb R^d$
\begin{eqnarray*}
C(z) &=& -\frac{1}{2}\langle z,Az\rangle + \int_{\mathbb R^{d}}\left(e^{i\langle x,z\rangle}-1-i\langle x,z\rangle h(x)\right) \tilde L(\rd x)+i\langle \tilde b,z\rangle \\
&&\quad+ \int_{\mathbb R^d} \left(e^{i\langle x,z\rangle}-1\right) \rho(\rd x) \\
&=&  -\frac{1}{2}\langle z,Az\rangle + \int_{\mathbb R^{d}}\left(e^{i\langle x,z\rangle}-1-i\langle x,z\rangle h(x)\right) g(x) L(\rd x)+ i\langle b,z\rangle \\
&&\quad - \int_{\mathbb R^{d}} i\langle x,z\rangle h(x)\left(1-g(x)\right)  L(\rd x) \\
&&\quad + \int_{\mathbb R^{d}} \left(e^{i\langle x,z\rangle}-1\right)\left(1-g(x)\right)  L(\rd x)\\
&=& -\frac{1}{2}\langle z,Az\rangle +  \int_{\mathbb R^{d}}\left(e^{i\langle x,z\rangle}-1-i\langle x,z\rangle h(x)\right) L(\rd x) + i\langle b,z\rangle,
\end{eqnarray*}
as required. The second part follows immediately from the fact that $V_t=0$ when $0\le t<T$. We now turn to the third part. Since $\tilde X$ and $T$ are independent, for any $B\in\mathfrak B(\mathbb R^d)$ we can use the second part to get
\begin{eqnarray*}
P(\tilde X_t\in B)P(T>t) = P(\tilde X_t\in B, T>t) = P(X_t\in B, T>t)\le P( X_t\in B).
\end{eqnarray*}
From here, the result follows by \eqref{eq: distrib of T}.
\end{proof}

\section{Rejection Sampling}\label{sec: sampling}

In this section, we set up a rejection sampling scheme for sampling from the tempered L\'evy process $\tilde X$, when we know how to sample from the original L\'evy process $X$. To begin with, assume that, for some $t>0$ we want to simulate $\tilde X_t$. For our approach to work, we need the distributions of both $\tilde X_t$ and $X_t$ to be absolutely continuous with respect to Lebesgue measure on $\mathbb R^d$. Thus, each distribution must have a probability density function (pdf). This always holds, for instance, if $A$ is an invertible matrix or if both
$$
\int_{\mathbb R^d}\left|\phi_{\mu}(z)\right|^t\rd z<\infty \mbox{ and }\int_{\mathbb R^d}\left|\phi_{\tilde\mu}(z)\right|^t\rd z<\infty.
$$
More delicate conditions, in terms of the corresponding L\'evy measures, can be found in Section 27 of \cite{Sato:1999}. 

Now assume that Assumptions A1 and A2 hold, and let $f_t$ and $\tilde f_t$ be the pdfs of $X_t$ and $\tilde X_t$, respectively. Since the inequality in the third part of Theorem \ref{thrm: repres} holds for all Borel sets, it holds for pdfs as well. Thus, for Lebesgue almost every $x$,
$$
\tilde f_t(x) \le e^{t\eta} f_t(x),
$$
where
$$
\eta = \rho(\mathbb R^d) = \int_{\mathbb R^d}\left(1-g(x)\right)L(\rd x).
$$
This means that we can set up a rejection sampling algorithm (see \cite{Devroye:1986}) to sample from $\tilde f_t$ as follows. \\

\noindent\textbf{Algorithm 1.}\\
\textbf{Step 1.} Independently simulate $U\sim U(0,1)$ and $Y\sim f_t$.\\
\textbf{Step 2.}  If $U\le e^{-\eta t}\tilde f_t(Y)/f_t(Y)$ accept, otherwise reject and go back to Step 1.\\

Let $p_t$ be the probability of acceptance on a given iteration and let $I_t$ be the expected number of iterations until the first acceptance. By a simple conditioning argument, it follows that $p_t = e^{-\eta t}$, and hence that $I_t=e^{\eta t}$. Note that both of these quantities approach $1$ as $t\to0$. On the other hand, when $t\to\infty$, we have $p_t\to0$ and $I_t\to\infty$. Thus, this method works best for small $t$. 

We now describe how to use Algorithm 1 to simulate the tempered L\'evy process $\tilde X=\{\tilde X_t:t\ge0\}$ on a finite mesh. For simplicity, assume that the mesh points are evenly spaced; the general case can be dealt with in a similar manner. Thus, fix $\Delta>0$ and assume that we want to simulate $\tilde X_{\Delta}, \tilde X_{2\Delta}, \dots, \tilde X_{n\Delta}$ for some $n\in\mathbb N$. To do this, we begin by simulating $n$ independent increments $Y_1,Y_2,\dots,Y_n\iid \tilde f_{\Delta}$ using Algorithm 1.  We expect this to take $nI_\Delta=ne^{\Delta \eta}$ iterations.  Now, to get values of the process, we set
\begin{eqnarray}\label{eq: get process}
\tilde X_{k\Delta} = \sum_{j=1}^k Y_j, \ \ k=1,2,\dots,n.
\end{eqnarray}

Next, consider the case, when we only want to simulate $\tilde X_t\sim \tilde f_t$ for some fixed $t>0$. While this can be done directly using Algorithm 1, when  $t$ is large, the expected number of iterations, $I_t=e^{\eta t}$, is large as well. Instead, we can choose some $n\in\mathbb N$ and sample $n$ independent increments $Y_{1},Y_{2},\dots,Y_{n}\iid\tilde f_{t/n}$. Then the sum $\tilde X_t=Y_{1}+Y_{2}+\cdots+Y_{n}$ has distribution $\tilde f_{t}$. In this case, we only expect to need $nI_{t/n}=ne^{\eta t/n}$ iterations. This requires fewer iterations so long as $n<e^{t\eta(1-1/n)}$. Since the function $b(x)=xe^{t\eta/x}$ is minimized (on $x>0$) at $x=t\eta$, it follows that the optimal choice of $n$ is near this value.  As an example, assume that we want to simulate one observation when $t=10$ and $\eta=1$. To do this directly, we expect to need $I_{10}= e^{10}\approx22026$ iterations. On the other hand, to simulate $10$ observations when $t=1$ and $\eta=1$ we only expect to need $10I_1 = 10*e^1\approx 27$ iterations. Thus, the second approach is much more efficient, in this case.

\section{Tempered Stable Distributions}\label{sec: ts}

Most, if not all, tempered L\'evy processes that have appeared in the literature, are those associated with tempered stable distributions. Before discussing these, we recall that an infinite variance stable distribution on $\mathbb R^d$ is an infinitely divisible distribution with no Gaussian part and a L\'evy measure of the form
\begin{eqnarray}\label{eq: Levy stable}
M_\alpha(B) = \int_{\mathbb S^{d-1}}\int_0^\infty 1_{B}(t\xi) r^{-1-\alpha} \rd r \sigma(\rd \xi), \ \ B\in\mathfrak B(\mathbb R^d),
\end{eqnarray}
where $\alpha\in(0,2)$ and $\sigma$ is a finite Borel measure on $\mathbb S^{d-1}$. For these distributions we will use the $h$-function
\begin{eqnarray}\label{eq: h for stable}
h_\alpha(x) = \left\{\begin{array}{ll}
0& \alpha\in(0,1)\\
1_{[|x|\le1]} & \alpha=1\\
1 & \alpha\in(1,2)
\end{array}\right..
\end{eqnarray}
We denote the distribution $ID(0,M_\alpha,b)_{h_\alpha}$ by $\mathrm S_\alpha(\sigma,b)$. For more on stable distributions, see the classic text  \cite{Samorodnitsky:Taqqu:1994}.

Now, following \cite{Rosinski:Sinclair:2010}, we define a tempered stable distribution on $\mathbb R^d$ as an infinitely divisible distribution with no Gaussian part and a L\'evy measure of the form
$$
\tilde M_\alpha(B) = \int_{\mathbb S^{d-1}}\int_0^\infty 1_{B}(t\xi) r^{-1-\alpha} q(r,\xi)\rd r \sigma(\rd \xi), \ \ B\in\mathfrak B(\mathbb R^d),
$$
where $\alpha\in(0,2)$, $q:(0,\infty)\times\mathbb S^{d-1}\mapsto[0,\infty)$ is a Borel function, and $\sigma$ is a finite Borel measure on $\mathbb S^{d-1}$. Here, the tempering function is $g(x) = q(|x|,x/|x|)$. Using the $h$-function given by \eqref{eq: h for stable}, we denote the distribution $ID(0,\tilde M_\alpha,b)_{h_\alpha}$ by $\mathrm{TS}_\alpha(\sigma,q,b)$. In this case, 
$$
\rho(B) = \int_{\mathbb S^{d-1}}\int_0^\infty 1_{B}(r\xi) r^{-1-\alpha}\left(1-q(r,\xi)\right) \rd r \sigma(\rd \xi), \ \ B\in\mathfrak B(\mathbb R^d),
$$
and Assumptions A1 and A2 become:
\begin{itemize}
\item[\textbf{B1.}] $0\le q(r,\xi) \le 1$ for all $r>0$ and $\xi\in\mathbb S^{d-1}$, and
\item[\textbf{B2.}] 
\begin{eqnarray*}
\int_{\mathbb S^{d-1}}\int_0^\infty r^{-1-\alpha}\left(1-q(r,\xi)\right) \rd r \sigma(\rd \xi)<\infty.
\end{eqnarray*}
\end{itemize}
Note that Assumption B2 implies that, for $\sigma$ almost every $\xi$, we have
\begin{eqnarray}\label{eq: q at 0}
\lim_{r\downarrow0}q(r,\xi)=1.
\end{eqnarray}

Next, we turn to the question of when both $\mathrm S_\alpha(\sigma,b)$ and $\mathrm{TS}_\alpha(\sigma,q,b)$ are absolutely continuous with respect to Lebesgue measure on $\mathbb R^d$. Before characterizing when this holds, we recall the following definition. The support of $\sigma$ is the collection of all points $\xi\in\mathbb S^{d-1}$ such that, for any open set $G\subset \mathbb S^{d-1}$ with $\xi\in G$, we have $\sigma(G)>0$. 

\begin{prop}\label{prop: abs cont} 
If the support of $\sigma$ contains $d$ linearly independent vectors, then both $\mathrm S_\alpha(\sigma,b)$ and $\mathrm{TS}_\alpha(\sigma,q,b)$ are absolutely continuous with respect to Lebesgue measure on $\mathbb R^d$.
\end{prop}

In particular, when $d=1$, this holds so long as $\sigma\ne0$.

\begin{proof}
By Proposition 24.17 and Theorem 27.10 in \cite{Sato:1999}, it suffices to show that, for $\sigma$ almost every $\xi\in\mathbb S^{d-1}$,
$$
\int_0^\infty r^{-1-\alpha} \rd r =\infty \mbox{ and } \int_0^\infty r^{-1-\alpha} q(r,\xi) \rd r =\infty.
$$
The first of these is immediately, while the second follows from \eqref{eq: q at 0}.
\end{proof}

In light of this proposition, we introduce the third assumption:
\begin{itemize}
\item[\textbf{B3.}] The support of $\sigma$ contains $d$ linearly independent vectors.
\end{itemize}

We now specialize our main results to the case of tempered stable distributions.

\begin{cor}\label{cor: main for ts}
Let $\mu=\mathrm S_\alpha(\sigma,b)$, $\tilde\mu=\mathrm{TS}_\alpha(\sigma,q,\tilde b)$, where 
$$
\tilde b = b - \int_{\mathbb S^{d-1}} \int_0^\infty  h_\alpha(r)\left(1-q(r,\xi)\right) r^{-\alpha} \rd r \xi \sigma(\rd \xi),
$$
If Assumptions B1 and B2 hold, then the results of Theorem \ref{thrm: repres} hold. If, in addition, Assumption B3 holds, then we can use Algorithm 1 to simulate from the L\'evy process associated with $\tilde\mu$.
\end{cor}

In particular, this means that, when Assumptions B1 and B2 hold, if $\{X_t:t\ge0\}$ is a L\'evy process with $X_1\sim\mu$ and $\{\tilde X_t:t\ge0\}$ is a L\'evy process with $\tilde X_1\sim\tilde\mu$, then $X_t=\tilde X_t$ for $0\le t< T$, where $T$ is as in Theorem \ref{thrm: repres}. This strengthens the well-known fact that tempered stable L\'evy processes behave like stable L\'evy processes in a short time frame, see \cite{Rosinski:Sinclair:2010}. Algorithm 1 requires the ability to sample from a stable distribution. In the one dimensional case, this is easily done using the classical method of  \cite{Chambers:Mallows:Stuck:1976}. In the multivariate case, this problem has not been fully resolved, however a method for simulating from a dense class of multivariate stable distributions is given in \cite{Nolan:1998}.

\section{$p$-Tempered $\alpha$-Stable Distributions}\label{sec: pts}

In the previous section, we allowed for tempered stable distributions with very general tempering functions. However, it is often convenient to work with families of tempering functions, which have additional structure. One such family, which is commonly used, corresponds to the case, where
\begin{eqnarray}\label{eq: q for p temp}
q(r,\xi) = \int_{(0,\infty)} e^{-sr^p} Q_\xi(\rd s).
\end{eqnarray}
Here $p>0$ and $\{Q_\xi:\xi\in\mathbb S^{d-1}\}$ is a measurable family of probability measures on $(0,\infty)$. For fixed $\alpha\in(0,2)$, the corresponding tempered stable distributions are called $p$-tempered $\alpha$-stable. For $p=1$ these were introduced in \cite{Rosinski:2007}, for $p=2$ they were introduced in \cite{Bianchi:Rachev:Kim:Fabozzi:2011}, and the general case was introduced in \cite{Grabchak:2012}. See also the recent monograph \cite{Grabchak:2016book}. 

Now, consider the distribution $\mathrm{TS}_\alpha(\sigma,q,b)$, where $q$ is as in \eqref{eq: q for p temp}, and define the measures
$$
Q(B) = \int_{\mathbb S^{d-1}}\int_{(0,\infty)} 1_{B}(s\xi) Q_\xi(\rd s)\sigma(\rd \xi), \ \ B\in\mathfrak B(\mathbb R^d)
$$
and
$$
R(B) = \int_{\mathbb R^{d}} 1_{B}\left(\frac{x}{|x|^{1+1/p}}\right) |x|^{\alpha/p}Q(\rd x), \ \ B\in\mathfrak B(\mathbb R^d).
$$
We call $R$ the Rosi\'nski measure of the distribution. For fixed $p>0$ and $\alpha\in(0,2)$, $R$ uniquely determines $q$ and $\sigma$. Further, Proposition 3.6 in \cite{Grabchak:2016book} implies that we can recover $\sigma$ by
\begin{eqnarray}\label{eq: back to sigma}
\sigma(B) = \int_{\mathbb R^d}1_B\left(\frac{x}{|x|}\right)|x|^\alpha R(\rd x), \ \ B\in\mathfrak B(\mathbb S^{d-1}).
\end{eqnarray}
Due to the importance of the Rosi\'nski measure, we sometimes denote the distribution $\mathrm{TS}_\alpha(\sigma, q,b)$, where $q$ is of the form \eqref{eq: q for p temp}, by $\mathrm{TS}_\alpha^p(R,b)$. We now characterize when Assumptions B1 and B2 hold.

\begin{prop}\label{prop: ets pts}
Consider the distribution $\mathrm{TS}_\alpha(\sigma,q,b)$, where $q$ is of the form \eqref{eq: q for p temp} and let $R$ be the corresponding Rosi\'nski measure. This distribution always satisfies Assumption B1. Further, it satisfies Assumption B2 if and only if $0<\alpha<p$ and $R$ is a finite measure on $\mathbb R^d$. In this case
$$
\eta =  \frac{\Gamma(1-\alpha/p)}{\alpha} R(\mathbb R^d).
$$
\end{prop}

\begin{proof}
The fact that Assumption B1 always holds follows from \eqref{eq: q for p temp} and the fact that $Q_\xi$ is a probability measure for every $\xi\in\mathbb S^{d-1}$. We now turn to Assumption B2. 
First assume that $0<\alpha<p$. In this case, the fact that $\{Q_\xi:\xi\in\mathbb S^{d-1}\}$ is a collection of probability measures implies that
\begin{eqnarray*}
 \int_0^\infty r^{-1-\alpha} \left(1-q(r,\xi)\right) \rd r &=& \int_{(0,\infty)}\int_0^\infty r^{-1-\alpha} \left(1-e^{-sr^p}\right)\rd r Q_\xi(\rd s)\\
&=& \int_{(0,\infty)}\frac{s^{\alpha/p}}{p}\int_0^\infty \left(1-e^{-r}\right) r^{-1-\alpha/p}\rd r Q_\xi(\rd s)\\
&=&\int_{(0,\infty)} s^{\alpha/p} Q_\xi(\rd s) \frac{\Gamma(1-\alpha/p)}{\alpha},
\end{eqnarray*}
where the second line follows by substitution and the third by integration by parts. It follows that
\begin{eqnarray*}
\eta &=& \int_{\mathbb S^{d-1}}\int_0^\infty r^{-1-\alpha}\left(1-q(r,\xi)\right) \rd r \sigma(\rd \xi)\\
&=& \frac{\Gamma(1-\alpha/p)}{\alpha} \int_{\mathbb S^{d-1}} \int_{(0,\infty)} s^{\alpha/p} Q_\xi(\rd s) \sigma(\rd \xi)\\
&=& \frac{\Gamma(1-\alpha/p)}{\alpha} \int_{\mathbb R^{d}} |x|^{\alpha/p} Q(\rd x)\\
&=& \frac{\Gamma(1-\alpha/p)}{\alpha} R(\mathbb R^d),
\end{eqnarray*}
which is finite if and only if $R(\mathbb R^d)<\infty$. Now assume that $\alpha\ge p$. We again have
\begin{eqnarray*}
 \int_0^\infty r^{-1-\alpha} \left(1-q(r,\xi)\right) \rd r 
&=& \int_{(0,\infty)}\frac{s^{\alpha/p}}{p}Q_\xi(\rd s)\int_0^\infty \left(1-e^{-r}\right) r^{-1-\alpha/p}\rd r.
\end{eqnarray*}
To see that  Assumption B2 does not hold in this case, observe that
$$
\int_0^\infty \left(1-e^{-r}\right) r^{-1-\alpha/p}\rd r \ge .5\int_0^1 r^{-\alpha/p}\rd r \ =\infty,
$$
where we use the fact that, for $r\in[0,1]$, $(1-e^{-r})\ge .5r$, see e.g.\ 4.2.37 in \cite{Abramowitz:Stegun:1972}.
\end{proof}

\begin{remark}
Perhaps, the most famous $p$-tempered $\alpha$-distributions are Tweedie distributions, which were introduced in \cite{Tweedie:1984}. These are sometimes also called classical tempered stable subordinators. These are one-dimensional distributions with $p=1$, $\alpha\in(0,1)$, $\sigma(\{-1\})=0$, $\sigma(\{1\})=a$, and $Q_{1}(\rd t) = \delta_c(\rd t)$, where $a,c>0$ and $\delta_c$ is the point-mass at $c$. In this case $R(\rd t) = ac^\alpha\delta_{1/c}(\rd t)$. It is not difficult to show that such distributions satisfy Assumptions B1, B2, and B3 and thus that Algorithm 1 can be used. Further, in this case, $\eta=\Gamma(1-\alpha)ac^\alpha/\alpha$, and it can be shown that $e^{-\eta t}f_t(Y)/\tilde f_t(Y) = e^{-cY}$. Thus, in this case, Algorithm 1 is computationally easy to perform. In fact, it reduces to the standard rejection sampling algorithm for Tweedie distributions given in e.g.\ \cite{Kawai:Masuda:2011}.
\end{remark}

\section{Simulations}\label{sec: sims}

In this section we perform a small scale simulation study to see how well Algorithm 1 works in practice. We focus on a parametric family of tempered stable distributions for which, up to now, there has not been an exact simulation method except for the inversion method. Specifically, we consider the family of $p$-tempered $\alpha$-stable distribution on $\mathbb R^1$ with Rosi\'nski measures of the form
$$
R_\ell(\rd x) = C (1+|x|)^{-2-\alpha-\ell}\rd x,
$$
where $\ell>0$ is a parameter and
$$
C = .5(\alpha+\ell+1)\frac{\alpha}{\Gamma(1-\alpha/p)}.
$$
These distributions were introduced in \cite{Grabchak:2016book} as a class of tempered stable distributions, which can have a finite variance, but still fairly heavy tails. If fact, if random variable $Y$ has a distribution of this type, then, for $\beta\ge0$,
$$
\rE|Y|^\beta<\infty \mbox{ if and only if } \beta<1+\alpha+\ell.
$$ 
Thus, $\ell$ controls how heavy the tails of the distribution are. Distributions with tails of this type are useful, for instance, in modeling financial returns, since, as is well-known, returns tend to have heavy tails, but the fact that they exhibit aggregational Gaussianity suggests that their tails cannot be too heavy and that the variance should be finite, see e.g.\ \cite{Cont:Tankov:2004} or  \cite{Grabchak:Samorodnitsky:2010}. 

Since $R_\ell$ is a finite measure, Proposition \ref{prop: ets pts} implies that, when $0<\alpha<p$, Assumptions B1 and B2 hold and
$$
\eta = \frac{\Gamma(1-\alpha/p)}{\alpha}R_\ell(\mathbb R) = 1.
$$
By \eqref{eq: back to sigma}, it follows that we are tempering a symmetric stable distribution with
$$
s:=\sigma(\{-1\})=\sigma(\{1\}) = C \int_0^\infty x^\alpha (1+x)^{-2-\alpha-\ell}\rd x.
$$
Let $\mu=\mathrm S_\alpha(\sigma,b)$ and $\tilde\mu=\mathrm{TS}_\alpha^p(R_\ell,\tilde b)$. When $p=1$, methods for evaluating the pdfs and related quantities of both of these distribution are available in the SymTS package \citep{Grabchak:Cao:2017} for the statistical software R. For this reason, we focus on the case $p=1$. In this case, we are restricted to $\alpha\in(0,1)$. 

\begin{figure}
\linespread{.75}
\begin{tabular}{cc}
\includegraphics[trim={1.15cm 1cm 1cm .5cm},clip,scale=.375]{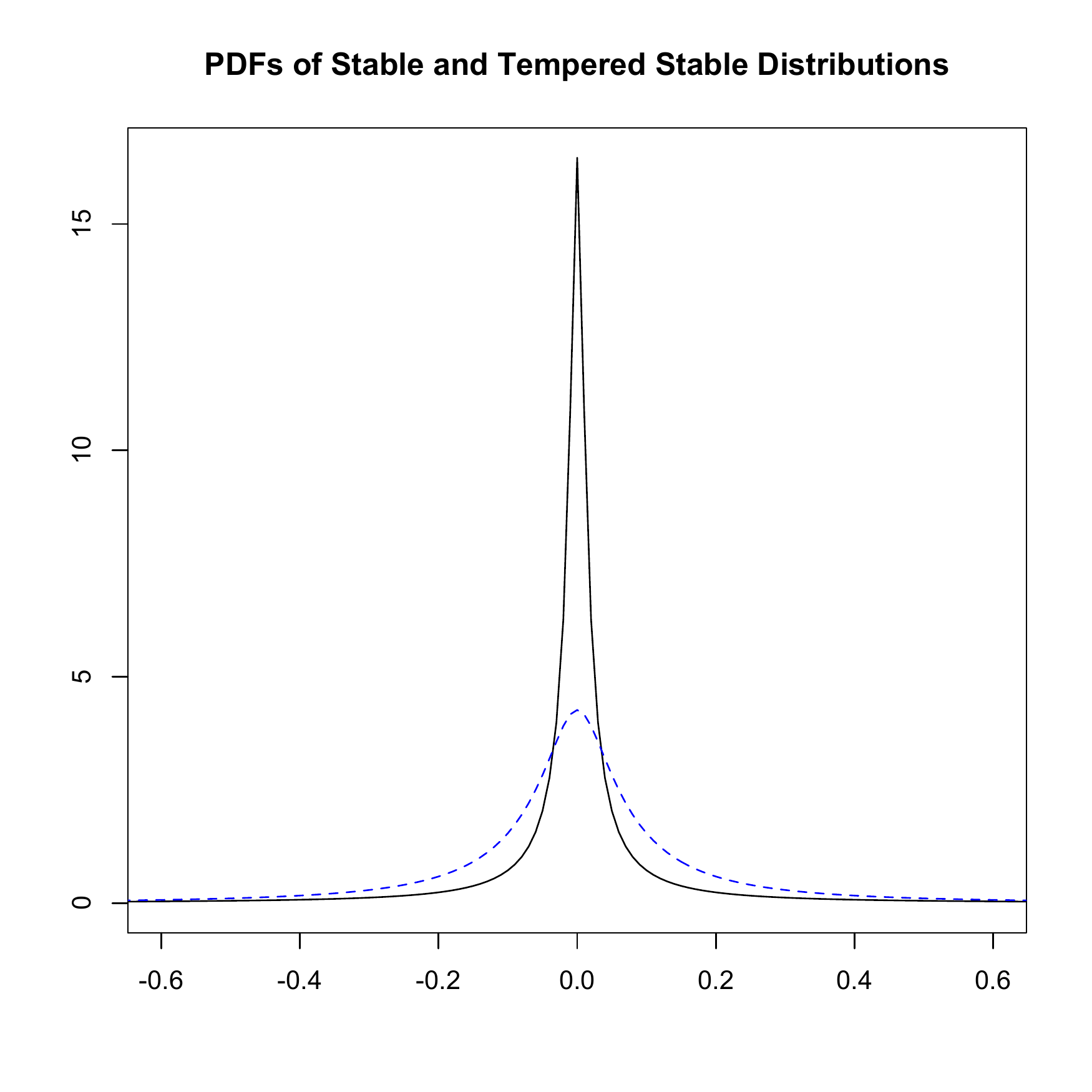} & \includegraphics[trim={1.15cm 1cm 1cm .5cm},clip,scale=.375]{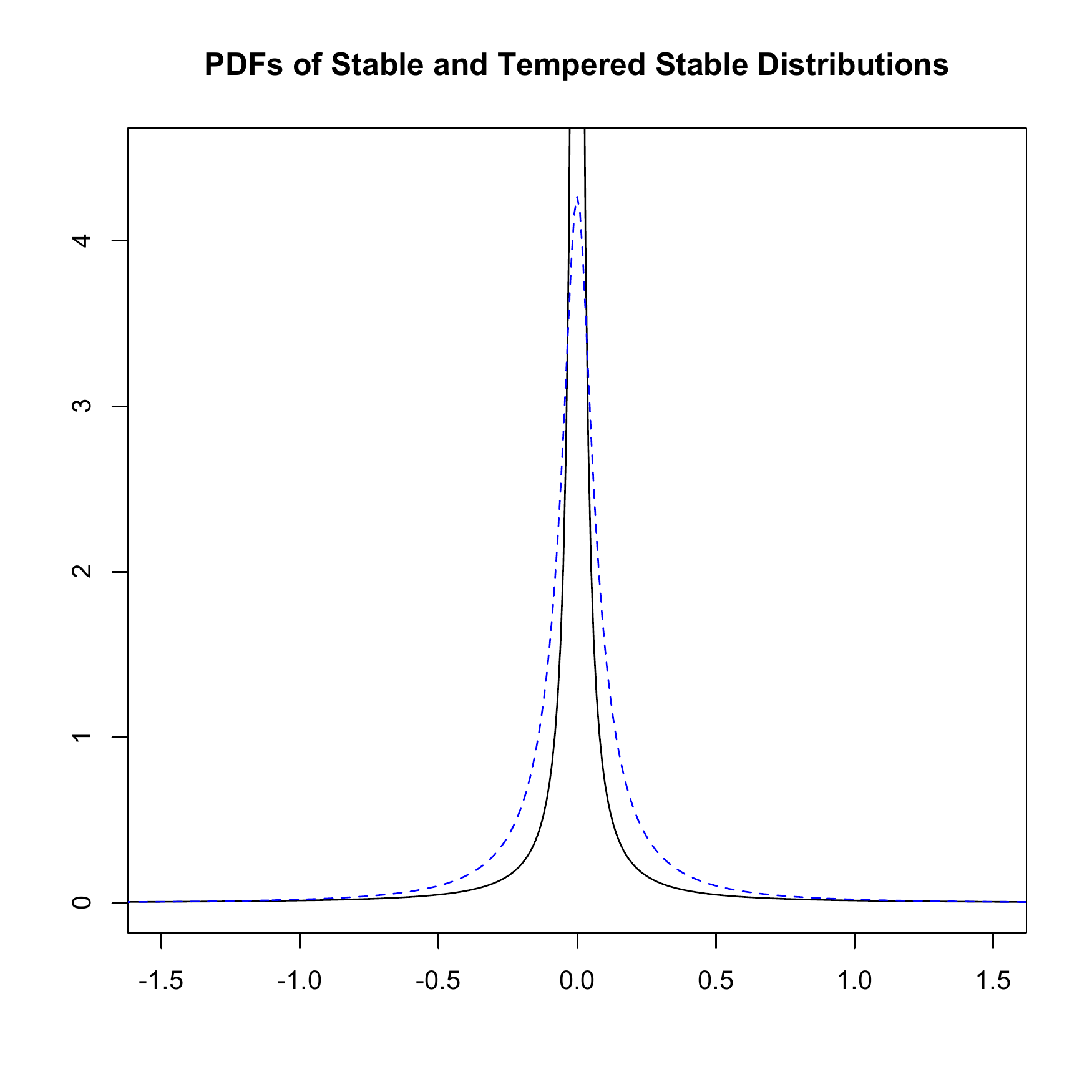} 
\end{tabular}
\vspace{-.5cm}\caption{The solid line is the pdf of the stable distribution $\mu$ and the dashed line is the pdf of the tempered stable distribution $\tilde\mu$. These are presented at two scales.}\label{fig: densities}
\end{figure}

\begin{figure}
\linespread{.75}
\begin{tabular}{cc}
\includegraphics[trim={1.15cm 1.5cm 1cm .5cm},clip,scale=.375]{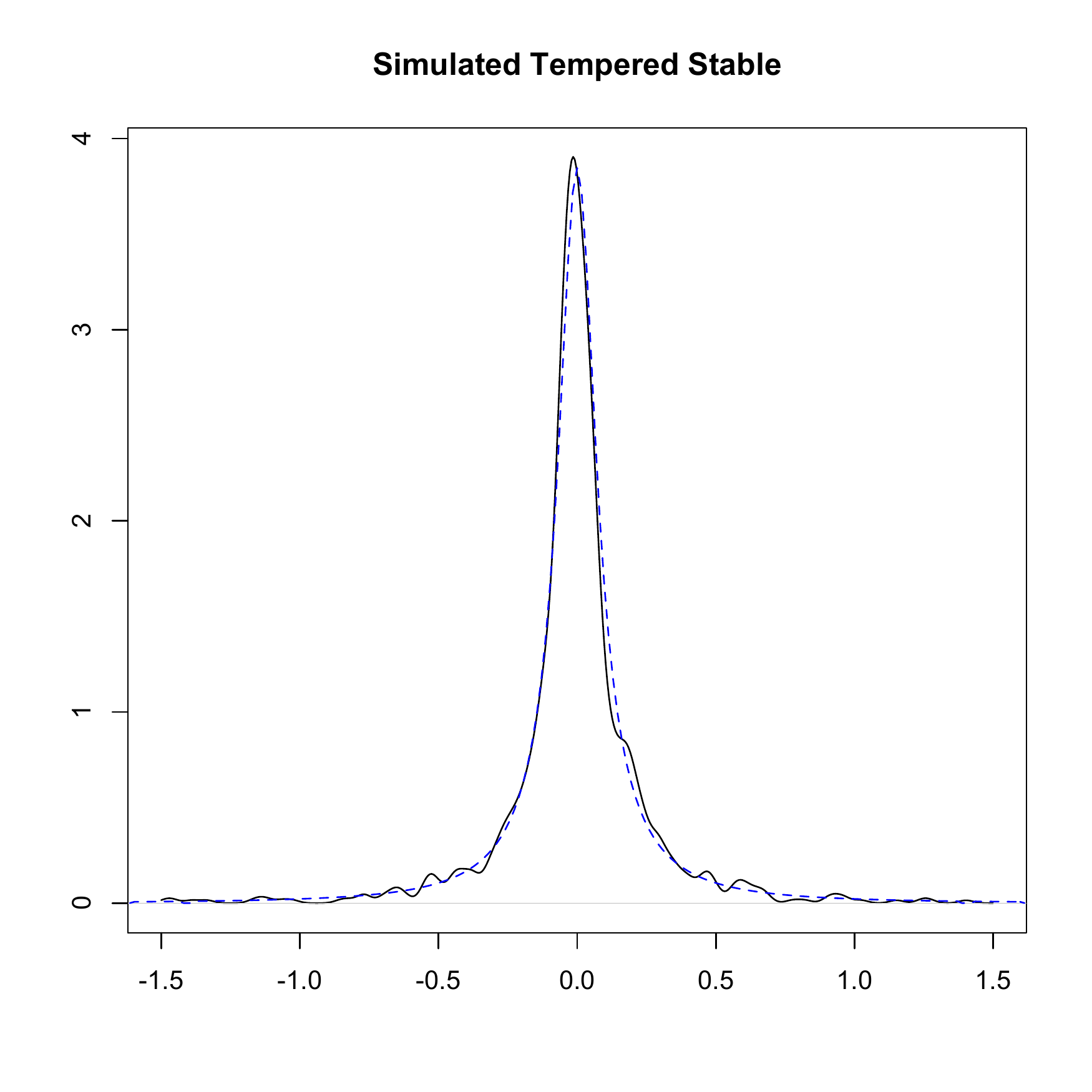} & \includegraphics[trim={1.15cm 1.5cm 1cm .5cm},clip,scale=.375]{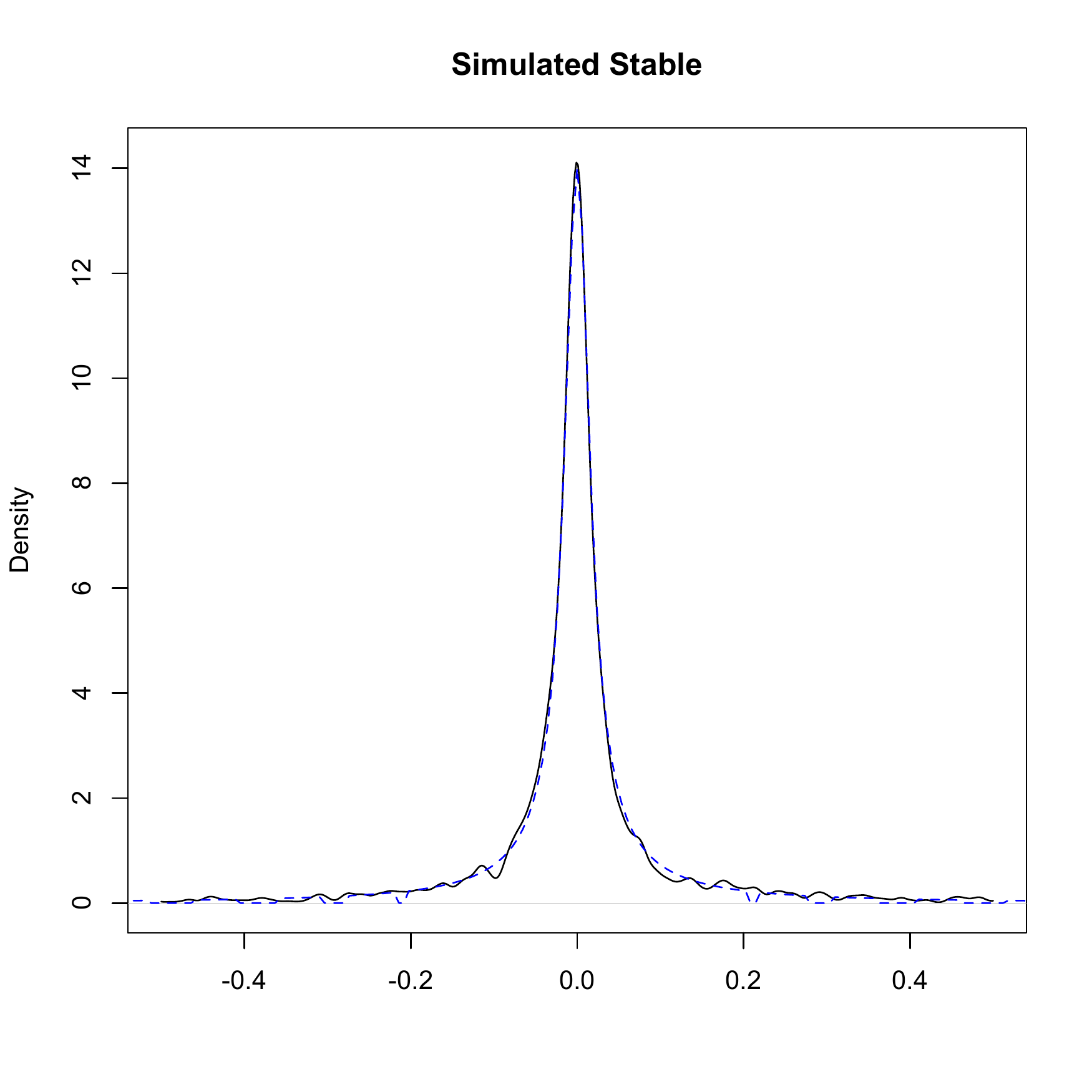}\\
\end{tabular}
\vspace{-.2cm}\caption{Tempered Stable and Stable KDE. On the left, the solid line is the KDE of the simulated tempered stable random variables, and the dashed line is the pdf of $\tilde\mu$ smoothed by the kernel and bandwidth used for the KDE. On the right, the solid line is the KDE of the simulated stable random variables, and the dashed line is the pdf of $\mu$ smoothed by the kernel and bandwidth used for the KDE.}\label{fig: hist}
\end{figure}

For our simulations, we took $\alpha=.75$, $\ell=1$, and $b=\tilde b=0$.  By numerical integration, we got $s=0.0591034$. Let $\tilde\mu=\mathrm{TS}_\alpha^p(R_\ell,\tilde b)=\mathrm{TS}_{.75}^1(R_{1},0)$ and let $\mu=\mathrm S_\alpha(\sigma, 0)=\mathrm S_{.75}(\sigma, 0)$, where $\sigma(\{-1\})=\sigma(\{1\}) = 0.0591034$. Plots of the pdfs of $\mu$ and $\tilde\mu$ are given in Figure \ref{fig: densities}. 

To get $1000$ observations from $\tilde\mu$ using Algorithm 1, we expect to need $1000 e^{1}\approx 2718$ iterations. For simplicity, we ran $3000$ iterations. Thus, we began by simulating $3000$ observations from $\mu$. This was done using the method of \cite{Chambers:Mallows:Stuck:1976}. We then applied Algorithm 1 to see which observations should be rejected. Here, we used the SymTS package to evaluate the pdfs. In the end, we wound up with $1110$ observations from $\tilde \mu$.

Figure \ref{fig: hist} (left) plots the kernel density estimator (KDE) for our samples from $\tilde\mu$. Here, KDE used a Gaussian kernel with a bandwidth of $.02463$. The plot is overlaid with the pdf of $\tilde\mu$ smoothed by the Gaussian Kernel with this bandwidth. This verifies, numerically, that we are simulating from the correct distribution.  For comparison, Figure \ref{fig: hist} (right), plots the KDE of the original $3000$ samples from $\mu$. This is overlaid with the smoothed pdf of $\mu$. Here, KDE used a Gaussian kernel with a bandwidth of $.006565$. 

\begin{figure}
\linespread{.75}
\begin{tabular}{cc}
\includegraphics[trim={1.15cm 1cm 1cm .5cm},clip,scale=.375]{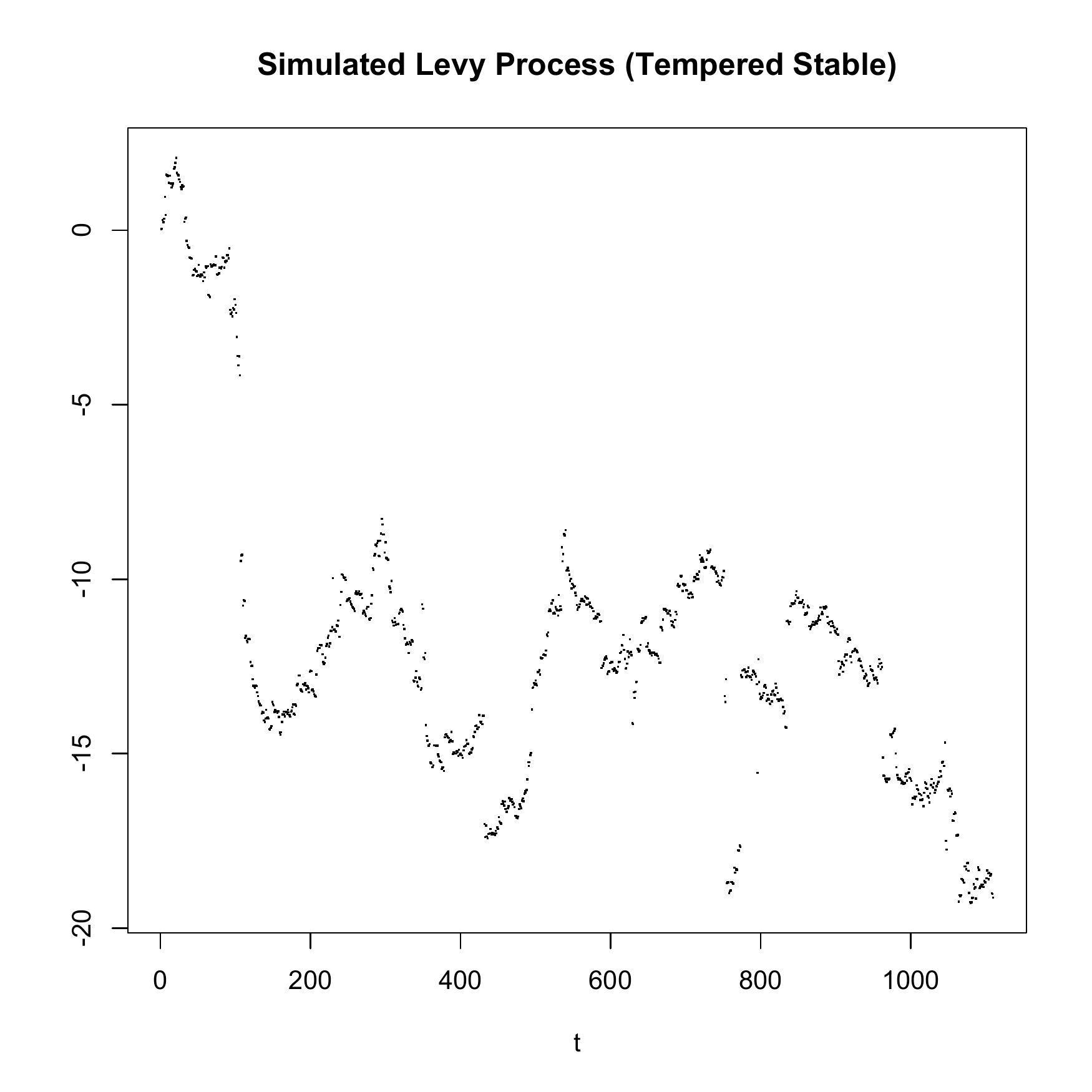} & \includegraphics[trim={1.15cm 1cm 1cm .5cm},clip,scale=.375]{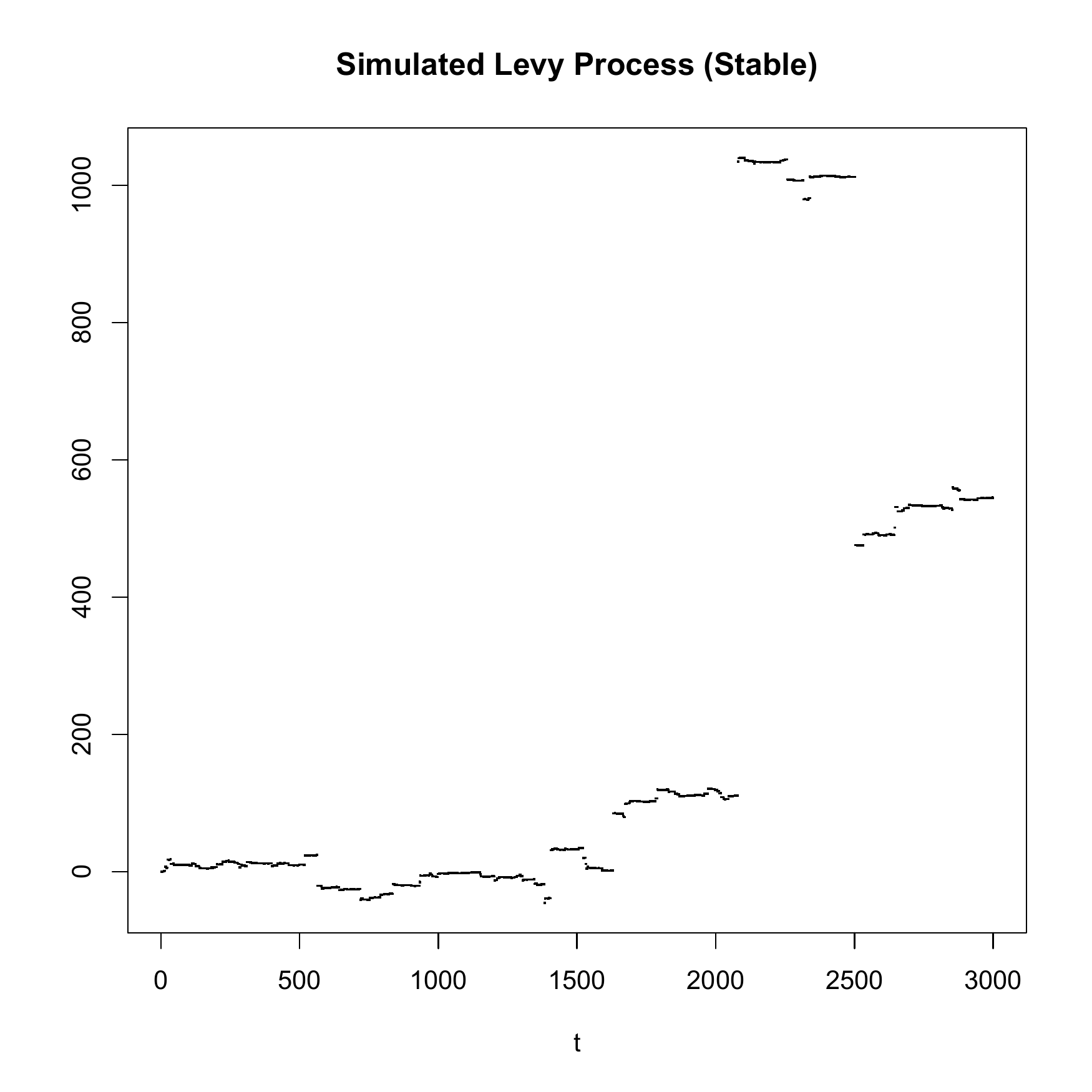}
\end{tabular}
\vspace{-.5cm}\caption{Simulated L\'evy Processes. On the left is the simulated tempered stable L\'evy process and on the right is the simulated stable L\'evy process. The tempered stable L\'evy process was obtained by rejecting some of the jumps of the stable L\'evy process.}\label{fig: process}
\end{figure}

\begin{figure}
\linespread{.75}
\begin{tabular}{cc}
\includegraphics[trim={1.15cm 1.1cm 1cm .5cm},clip,scale=.375]{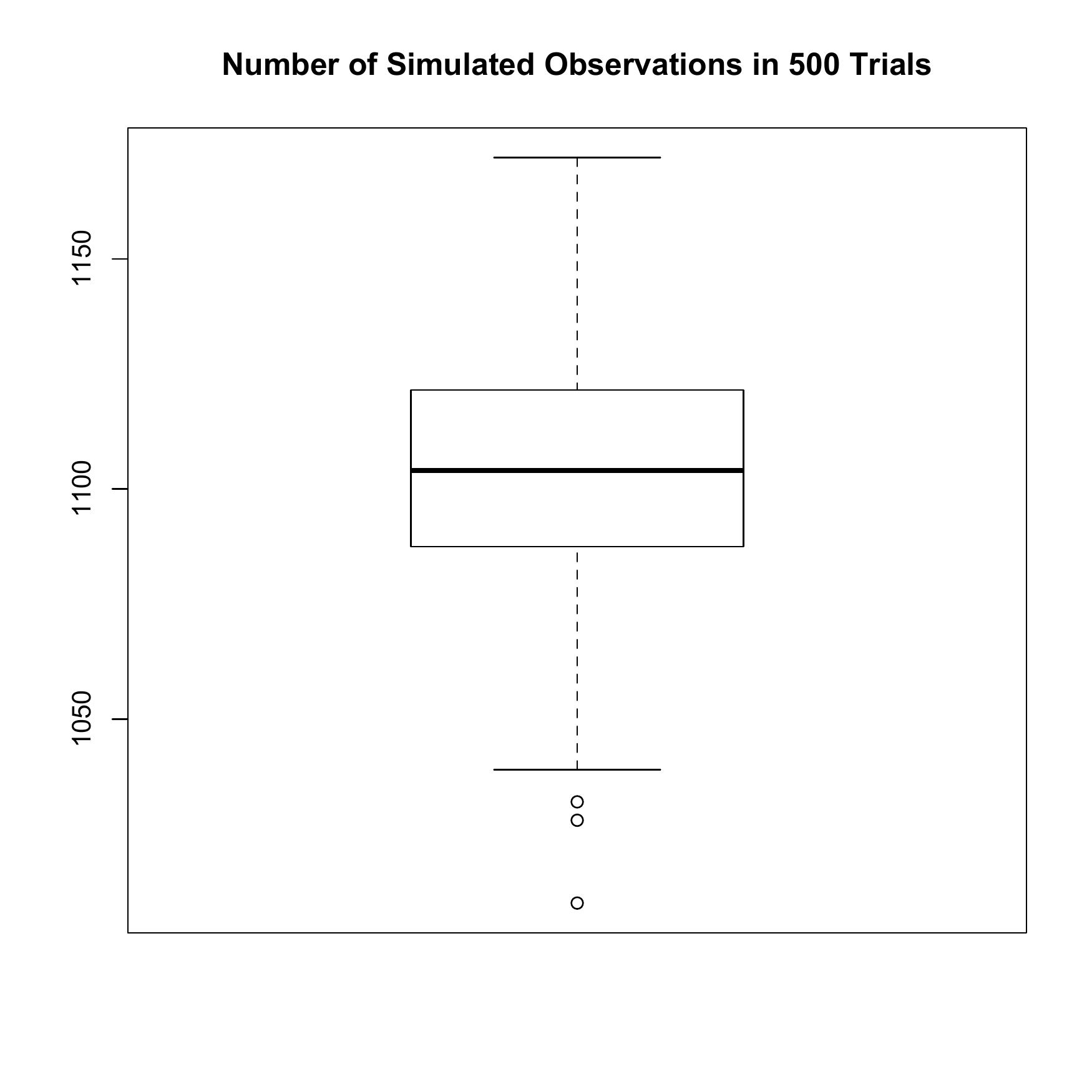} & \includegraphics[trim={1.15cm 1.1cm 1cm .5cm},clip,scale=.375]{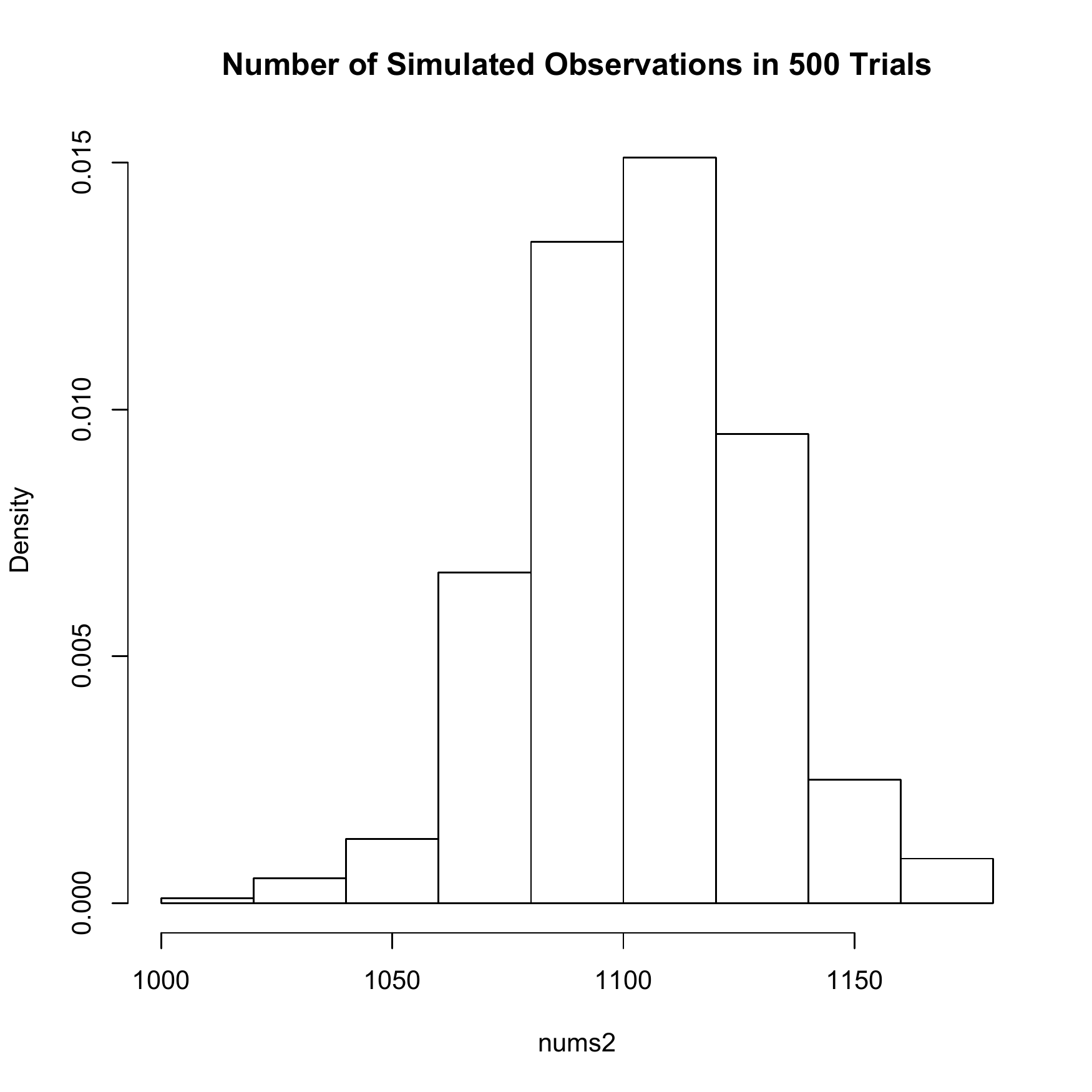}
\end{tabular}
\vspace{-.5cm}\caption{The number of observations from $\tilde\mu$ (obtained based on a sample of size $3000$ from $\mu$) was evaluated $500$ times. A boxplot and histogram of these is given.}\label{fig: boxplot}
\end{figure}

Next, consider the L\'evy process $\{\tilde X_t:t\ge0\}$, where $\tilde X_1\sim\tilde\mu$. To simulate $\tilde X_1,\tilde X_2,\dots,\tilde X_{1110}$, we apply \eqref{eq: get process} to the iid increments, which we simulated above. A plot of this L\'evy process is given in the left plot of Figure \ref{fig: process}. For comparison, the right plot of Figure \ref{fig: process} gives the L\'evy process based on the original $3000$ iid increments from $\mu$. Comparing the two processes, we see that all of the largest jumps have been rejected.

In our simulation, we used $3000$ observations from $\mu$ to get $1110$ observations from $\tilde\mu$. Of course, due to the structure of the algorithm, if we run the simulation again, we may get a different number of observations from $\tilde\mu$. We performed the simulation $500$ times to see how many observations from $\tilde\mu$ are obtained.  Figure \ref{fig: boxplot} gives a boxplot and a histogram for all resulting values. The smallest value observed was $1010$ and the largest was $1172$. The mean was $1104.1$ with a standard deviation of $25.1$. Note that the observed mean is very close to the theoretical mean of $3000e^{-1}\approx 1103.6$.

Now, consider the case, where we want to simulate from $\tilde\mu^{10}=\mathrm{TS}(10R,0)$, which is the distribution of $\tilde X_{10}$. A plot of this distribution is given in Figure \ref{fig: t10} (left). We have two choices. First, we can use Algorithm 1 directly, which requires, on average, $e^{10}\approx22026$ iterations to get one observation. Second, we can use Algorithm 1 to simulate $10$ independent observations from $\tilde \mu$, and then aggregate these by taking the sum. This requires, on average, $10e^1\approx27$ iteration to get one observation. Clearly, the second approach is more efficient, and, hence, it is the one that we use. 

\begin{figure}
\linespread{.75}
\begin{tabular}{cc}
\includegraphics[trim={1.15cm 1cm .5cm .5cm},clip,scale=.375]{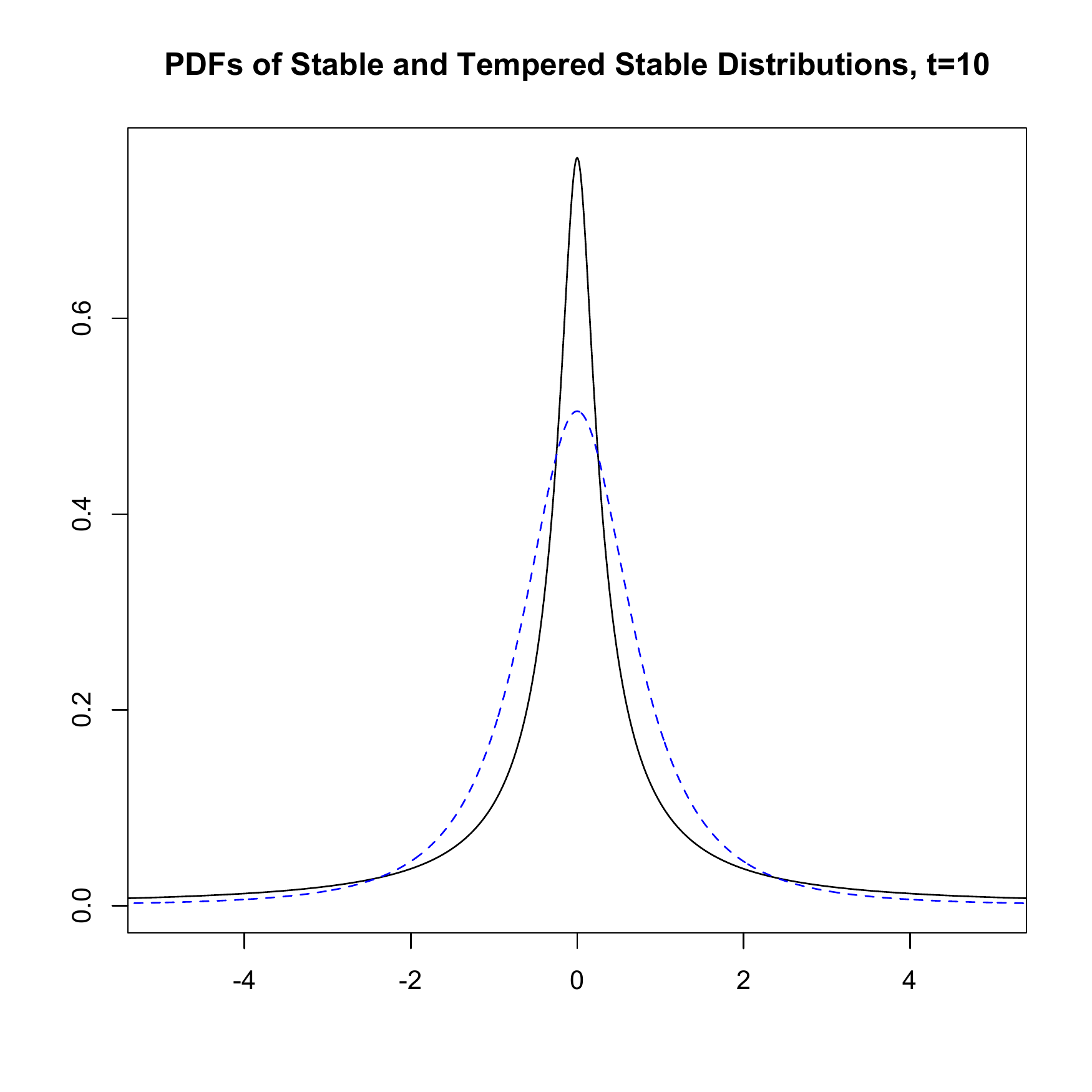} & \includegraphics[trim={1.15cm 1cm 1cm .5cm},clip,scale=.375]{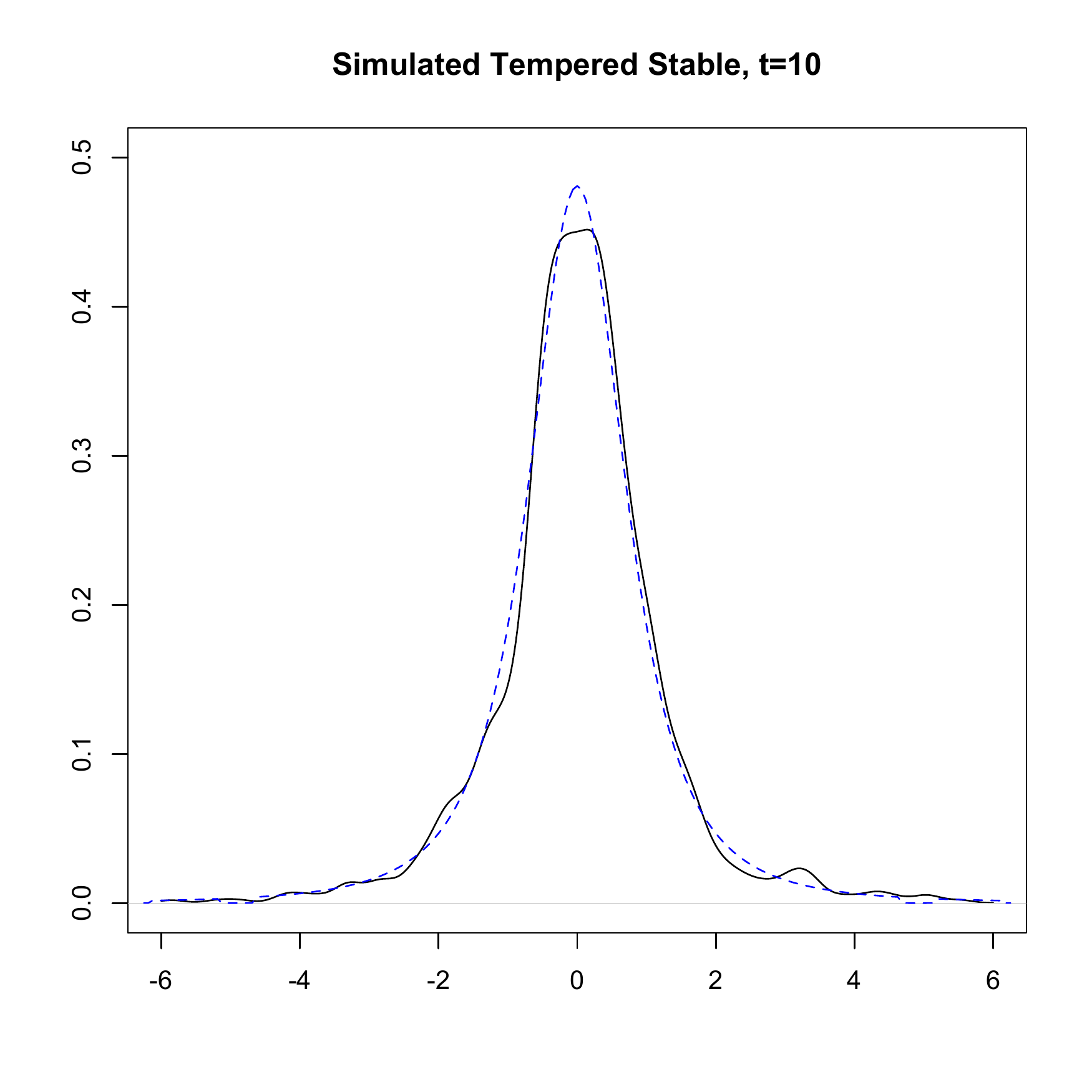}
\end{tabular}
\vspace{-.5cm}\caption{Plots for $t=10$. On the left, the solid line is the pdf of the stable distribution $\mu^{10}$ and the dashed line is the pdf of the tempered stable distribution $\tilde\mu^{10}$. On the right, the solid line is the KDE of the simulated tempered stable random variables, and the dashed line is the pdf of $\tilde\mu^{10}$ smoothed by the kernel and bandwidth used for the KDE.}\label{fig: t10}
\end{figure}

Say that we want $1000$ independent observations from $\tilde\mu^{10}$. This requires $10000$ independent observations from $\tilde\mu$, and we expect to need $10000e^1\approx27183$ iterations. We performed $30000$ iterations, which gave us $11130$ observations from $\tilde \mu$. These, in turn, gave $1113$ observations from $\tilde\mu^{10}$. In Figure \ref{fig: t10} (right), we plot the KDE of the simulated data. This is  overlaid with the pdf of $\tilde\mu$ smoothed by the appropriate kernel and bandwidth. Here, KDE used the Gaussian kernel with bandwidth $0.1863$.

\begin{table}
\begin{center}
\begin{tabularx}{\textwidth}{c|ccc|cc|c}
\hline
 & \multicolumn{3}{c|}{Algorithm 1} & \multicolumn{2}{c|}{Inversion Method} & \multicolumn{1}{c}{} \\
\hline
t & run time & iterations & obs &run time& obs & ratio\\
\hline
1 & 24.711 sec & 3000 &1121 &  285.210 sec & 1000  & 0.087 \\
2 & 47.399 sec & 6000 & 1105 & 255.879  sec & 1000&  0.185\\
5 & 118.563 sec & 15000 & 1104 & 224.227 sec & 1000&  0.529\\
10 & 237.747 sec & 30000  & 1084 &  193.319 sec &1000& 1.230 \\
20 &443.286  sec & 56000 & 1024 & 191.004 sec &1000& 2.321\\
\hline
\end{tabularx}
\end{center}
\caption{Comparison of the performance of Algorithm 1 and the inversion method for several values of $t$. For Algorithm 1, the iterations column gives the number of stable random variables simulated and the obs column gives the number of tempered stable random variables obtained after performing rejection and aggregation. The ratio column is the run time for Algorithm 1 divided by the run time for the inversion method.}\label{table}
\end{table}

\begin{table}
\begin{center}
\begin{tabularx}{\textwidth}{cc|cc|cc|c}
\hline
 & & \multicolumn{2}{c|}{Algorithm 1} & \multicolumn{2}{c|}{Inversion Method} & \multicolumn{1}{c}{}  \\
\hline
$\alpha$ & $\ell$ & run time & observations &run time& observations & ratio\\
\hline
.50 & 0.5 & 66.841 sec &1057 &  785.870 sec & 1000 & 0.085 \\
.50 & 1.0 & 30.531 sec & 1116 & 338.867  sec & 1000 &0.090 \\
.50 & 5.0 & 47.035 sec  &  1090 &  509.506 sec & 1000 & 0.092\\
\hline
.75 & 0.5& 53.346 sec   & 1124 &  668.588 sec & 1000 & 0.080\\
.75 & 1.0& 23.844  sec & 1082 &  282.788 sec & 1000 &0.084\\
.75 & 5.0& 37.353 sec  & 1093  &  448.820 sec& 1000 & 0.083\\
\hline
.95 & 0.5& 51.121 sec  & 1101 & 601.804 sec & 1000& 0.084\\
.95 & 1.0& 23.444 sec  &  1096 &  278.804 sec& 1000&0.084 \\
.95 & 5.0& 38.190 sec  & 1100 & 470.784 sec& 1000& 0.081\\
\hline
\end{tabularx}
\end{center}
\caption{Comparison of the performance of Algorithm 1 and the inversion method for several values of $\alpha$ and $\ell$. For Algorithm 1, the observations column gives the number of tempered stable random variables obtained based on a sample for $3000$ observations from the stable distribution. The ratio column is the run time for Algorithm 1 divided by the run time for the inversion method.
}\label{table2}
\end{table}

For a comparison, we also performed simulations using the inversion method, which is implemented in the SymTS package. First, we fixed $\alpha=0.75$, $\ell=1.0$, and considered $t=1, 2, 5,10, 20$. Computation times for both approaches are reported in Table \ref{table}. Not surprisingly, when $t$ is small, Algorithm 1 is more efficient, but for large $t$ the inversion method works better. Second, we fixed $t=1$ and considered $\alpha=0.5, 0.75, 0.95$ and $\ell=0.5, 1.0, 5.0$. Computation times are reported in Table \ref{table2}. Both sets of simulations were performed on a desktop PC with a 3.40GHz Intel Core i7-6700 CPU. The computer was running Ubuntu 16.04.4 LTS and R version 3.2.3. 

For the above simulations, we attempted to cover a large  part of the parameter space. However, while the parameter $\alpha$ can, in principle, take any value in $(0,1)$, we did not consider the case $\alpha<0.5$. This is because we found it difficult to numerically evaluate the pdfs with enough accuracy in this case. We conjecture that this difficulty is due to the fact that, for such values of $\alpha$, stable distributions are extremely heavy tailed. On the other hand, there is no difficulty with values near $1$, as is illustrated by the simulations for $\alpha=.95$.

For both Algorithm 1 and the inversion method, the most expensive part of the calculation involves evaluating the required functions. These are $f_t$ and $\tilde f_t$ for Algorithm 1 and the quantile function for the inversion method. When we are simulating many observations from the same distribution, both methods can be improved by precomputing these functions on a grid and then using interpolation.

\end{document}